\documentclass[a4paper,reqno,11pt]{amsart}
\usepackage[T1]{fontenc}
\usepackage[utf8]{inputenc}
\usepackage{amssymb, amsmath, amsthm, graphicx, enumerate, enumitem, color}
\usepackage[usenames,dvipsnames,svgnames,table]{xcolor}
\usepackage[foot]{amsaddr}

\makeatletter
\def\thm@space@setup{
\thm@preskip=4mm
\thm@postskip=0mm
}
\makeatother

\usepackage{hyperref}
\usepackage{xcolor}
\hypersetup{
    colorlinks,
    linkcolor={red!50!black},
    citecolor={blue!50!black},
    urlcolor={blue!80!black}
}

\usepackage{cleveref}

\usepackage{mathtools}

\usepackage{thmtools, thm-restate}

\usepackage{textpos}

\newtheorem{theorem}{Theorem}
\newtheorem{lemma}[theorem]{Lemma}
\newcommand{\td}{\operatorname{td}}

\let\le\leqslant
\let\ge\geqslant

\let\geq\geqslant

\let\epsilon\varepsilon

\makeatletter
\let\old@setaddresses\@setaddresses
\def\@setaddresses{\bigskip\bgroup\parindent 0pt\let\scshape\relax\old@setaddresses\egroup}
\makeatother

\title{Treedepth vs circumference}

\author[M.~Briański]{Marcin Briański}
\address[M.~Briański]{Theoretical Computer Science Department\\
  Faculty of Mathematics and Computer Science\\
  Jagiellonian University\\
  Kraków, Poland}
\email{marcin.brianski@doctoral.uj.edu.pl}

\author[G.~Joret]{Gwena\"el Joret}
\address[G.~Joret]{Computer Science Department \\
  Universit\'e libre de Bruxelles\\
  Brussels, Belgium}
\email{gwenael.joret@ulb.be}

\author[K.~Majewski]{Konrad Majewski}
\address[K.~Majewski]{Institute of Informatics\\
  Faculty of Mathematics, Informatics and Mechanics\\
  University of Warsaw\\
  Poland}
\email{k.majewski@mimuw.edu.pl}

\author[P.~Micek]{Piotr Micek}
\address[P.~Micek]{Theoretical Computer Science Department\\
  Faculty of Mathematics and Computer Science\\
  Jagiellonian University\\
  Kraków, Poland}
\email{piotr.micek@uj.edu.pl} 

\author[M.~T.~Seweryn]{Michał T.\ Seweryn}
\address[M.~T.~Seweryn]{Theoretical Computer Science Department\\
  Faculty of Mathematics and Computer Science\\
  Jagiellonian University\\
  Kraków, Poland}
\email{michal.seweryn@tcs.uj.edu.pl}

\author[R.~Sharma]{Roohani Sharma}
\address[R.~Sharma]{Max Planck Institute for Informatics \\
  Saarland Informatics Campus\\
  Saarbr\"ucken\\
  Germany}
\email{rsharma@mpi-inf.mpg.de}

\thanks{G.\ Joret is supported by a CDR grant from the Belgian National Fund for Scientific Research (FNRS), a PDR grant from FNRS, and by the Wallonia Brussels International (WBI) agency. 
This work is a part of project BOBR (K.~Majewski) that has received funding from the European Research Council (ERC) under the European Union’s Horizon 2020 research and innovation programme (grant agreement No. 948057).
P.\ Micek, M.\ T.\ Seweryn and M.\ Briański are supported by the National Science Center of Poland under grant UMO-2018/31/G/ST1/03718 within the BEETHOVEN program.}

\begin{document}

\begin{textblock}{20}(-1.4, 9.4)
	\includegraphics[width=30px]{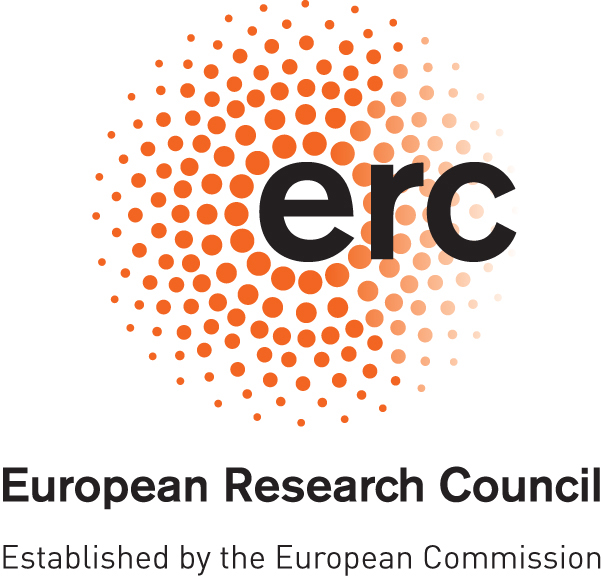}%
\end{textblock}
\begin{textblock}{20}(-1.4, 10)
	\includegraphics[width=30px]{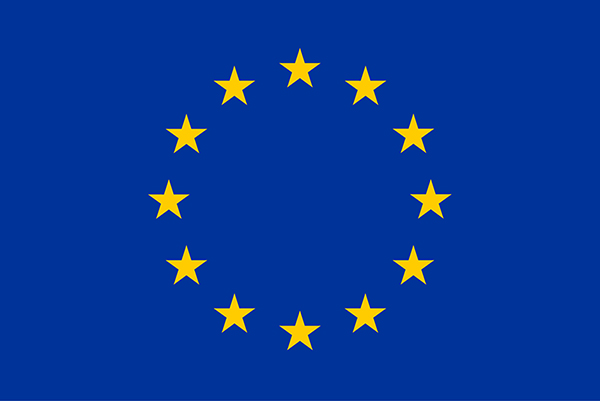}%
\end{textblock}

\begin{abstract}
The circumference of a graph $G$ is the length of a longest cycle in $G$, or $+\infty$ if $G$ has no cycle. 
Birmelé (2003)  showed that the treewidth of a graph $G$ is at most its circumference minus $1$. 
We strengthen this result for $2$-connected graphs as follows: 
If $G$ is $2$-connected, then its treedepth is at most its circumference. 
The bound is best possible and improves on an earlier quadratic upper bound due to Marshall and Wood (2015). 
\end{abstract}

\maketitle

\section{Introduction}

Treewidth is an important graph invariant in structural and algorithmic graph theory. 
We refer the reader to the textbook by Diestel~\cite{Diestel5thEdition} for its definition and for background on the topic. 
A classical  upper bound on the treewidth of a graph is the following. 

\begin{theorem}[Birmelé~\cite{B03}]\label{thm:tw-vs-circumference}
  Every graph has treewidth at most its circumference minus $1$. 
\end{theorem}

Here, the \emph{circumference} of a graph \(G\) is the length of a longest cycle in \(G\), or $+\infty$ if $G$ has no cycle. 
This bound is best possible, as witnessed by complete graphs. 
In this paper, we show that `treewidth' can be replaced with `treedepth' in the above theorem, provided that $G$ is $2$-connected:

\begin{theorem}\label{thm:td-vs-circumference}
  Let $G$ be a \(2\)-connected graph. 
  Then, each edge of $G$ is contained in a cycle whose length is at least the treedepth of $G$. 
  Consequently, $G$ has treedepth at most its circumference. 
\end{theorem}

Treedepth is a key invariant in the `sparsity theory' for graphs initiated by Ne{\v{s}}et{\v{r}}il and Ossona de Mendez~\cite{NOdM-book}, with several algorithmic applications. 
It is defined as follows. 
An \emph{elimination forest} of a graph \(G\) is a rooted forest consisting of trees \(T_1, \ldots, T_p\) such that
the sets \(V(T_1), \ldots, V(T_p)\) partition the set \(V(G)\) and for each edge \(xy \in E(G)\), the vertices
\(x\) and \(y\) belong to one tree \(T_i\) and in that tree one of them is an ancestor of the other.
The \emph{vertex-height} of an elimination forest is the maximum number of vertices on a root-to-leaf path in
any of its trees. The \emph{treedepth} of a graph \(G\), denoted \(\td(G)\), is the minimum vertex-height of an elimination forest of \(G\), and an elimination forest realizing this minimum is called \emph{optimal}. 

It is well-known that every graph has treewidth at most its treedepth minus $1$ (see e.g.\ \cite{BGHK95}).  
Also, it is known that the treewidth of a graph is equal to the maximum treewidth of its blocks (see Section~\ref{sec:prelim} for the definition of blocks). 
Hence, Theorem~\ref{thm:td-vs-circumference} implies Theorem~\ref{thm:tw-vs-circumference}.  
Theorem~\ref{thm:td-vs-circumference} is also an improvement on earlier quadratic bounds on the treedepth, as we now explain. 
Dirac~\cite{D52} proved that, in a $2$-connected graph with circumference $k$, every path has at most $\lceil \frac{k^2}{2}\rceil$ vertices. 
This in turn implies an upper bound of $\lceil \frac{k^2}{2}\rceil$ on the treedepth of such a graph. 
The latter bound on treedepth was subsequently improved to $\lfloor\frac{k}{2}\rfloor(k-1)+1$ by Marshall and Wood~\cite{MW13}. 
The bound in Theorem~\ref{thm:td-vs-circumference} is best possible, again because of complete graphs. 
The $2$-connectivity assumption is essential, because otherwise treedepth is not bounded by any function of the circumference: if $G$ is an $n$-vertex path to which an edge is added to create a triangle, then $G$ has circumference $3$ but treedepth at least $\log_2 n$.

We conclude this introduction with an algorithmic application of our result. 
Given an integer $d\geq 1$, 
Chen {\it et al.}~\cite{ChenEtAl21} designed a data structure for maintaining an optimal elimination forest of a dynamic graph $G$ with worst case $2^{O(d^2)}$ update time, under the promise that the treedepth of $G$ never exceeds $d$. 
Here, the graph $G$ is {\em dynamic} in the sense that edges can added or removed, one at a time. 
An update time of $2^{O(d^2)}$ is a natural barrier in this context, because the best known FPT algorithm for deciding whether an $n$-vertex graph $G$ has treedepth at most $d$ runs in time $2^{O(d^2)} \cdot n$, see~\cite{ReidlEtAl14}. 
Any improvement on the $2^{O(d^2)}$ update time would lead to a corresponding improvement of the latter result, by adding all edges of $G$ one at a time.  
One application of the above result that is developed in~\cite{ChenEtAl21} is as follows: Given an integer $k\geq 1$, there is a data structure for 
answering queries of the following type on a dynamic graph $G$: {\it Does $G$ contain a cycle of length at least $k$?} 
Their data structure answers these queries in constant time and has an amortized update time of $2^{O(k^4)} + O(k \log n)$, assuming access to a dictionary on the edges of $G$.  
(Note that here $G$ is not required to have treedepth at most $k$, it is an arbitrary graph.) 
As can be checked in their proof, it turns out that the $2^{O(k^4)}$ term in the latter result comes from their $2^{O(d^2)}$ bound mentioned previously combined with the fact that the treedepth $d$ of a $2$-connected graph with circumference $k$ is $O(k^2)$. 
Using Theorem~\ref{thm:td-vs-circumference} instead in their proof  reduces the amortized update time down to $2^{O(k^2)} + O(k \log n)$.  
This solves an open problem from~\cite{ChenEtAl21}.

\section{Preliminaries}
\label{sec:prelim}
Let \(G\) be a connected graph with at least two vertices.
A vertex \(x\) in $G$ is called a \emph{cutvertex} of \(G\) if \(G - x\) is disconnected,
and an edge \(e\) in $G$ is called a \emph{bridge} of \(G\) if \(G - e\) is disconnected.
The graph \(G\) is \emph{\(2\)-connected} if \(G\) does not have a cutvertex and \(G\) has at least three vertices.
A \emph{block} of \(G\) is a maximal connected subgraph \(B \subseteq G\)
such that \(B\) does not have a cutvertex. Each block of \(G\) is either a bridge (together with its ends),
or a \(2\)-connected subgraph of \(G\).
Every vertex of \(G\) that is not a cutvertex of \(G\) belongs to exactly one block of \(G\).
The \emph{block tree} of \(G\) is the tree whose nodes are the cutvertices and blocks of \(G\),
and in which two nodes are adjacent if and only if one of them is a cutvertex \(x\) and the other
is a block \(B\) such that $x$ is a vertex of $B$.

Note that \(\td(G) \le \td(G - x) + 1\)  for every \(x \in V(G)\), since we can obtain an elimination forest
of \(G\) from an optimal elimination forest of \(G - x\) by attaching \(x\) as a common root above all trees
in the forest. Moreover, a graph has treedepth at most \(2\) if and only if each of its components is a star,
i.e.\ a graph isomorphic to \(K_{1, n}\) for some \(n \ge 0\). Hence, every \(2\)-connected graph has treedepth
at least \(3\).

\section{The Proof}
\begin{lemma}\label{lem:path-in-block-tree}
  Let \(G\) be a connected graph on at least two vertices and let \(x_0 \in V(G)\).
  Then, for some \(m \ge 0\), there exist blocks \(B_0, \ldots, B_m\) and cutvertices \(x_1, \ldots, x_m\) of \(G\)
  such that \(B_0 x_1 B_1 x_2 \cdots x_m B_m\) is a path in the block tree of \(G\) with \(x_0 \in V(B_0)\) and
  \[
    \sum_{i=0}^m \td(B_i - x_i) \ge \td(G - x_0).
  \]
\end{lemma}
\begin{proof}
  Let \(\mathcal{T}\) denote the block tree of \(G\).
  We prove the lemma by induction on the number of nodes in \(\mathcal{T}\).
  In the base case, \(\mathcal{T}\) has just one node corresponding to the unique block \(B_0 = G\), and the trivial path \(B_0\) satisfies the lemma.
  For the inductive step, assume that \(\mathcal{T}\) has more than one node.
  We split the argument depending on whether \(x_0\) is a cutvertex of \(G\) or not.

  First suppose that \(x_0\) is a cutvertex of \(G\).
  Let \(\mathcal{T}_1,\ldots, \mathcal{T}_\ell\) denote the components of the forest \(\mathcal{T} - x_0\).
  For each \(j\in\{1, \ldots, \ell\}\), let \(G_j\) denote the union of all
  blocks of $G$ that are vertices in \(\mathcal{T}_j\). Note that each \(\mathcal{T}_j\) is the block tree of \(G_j\)
  and has less nodes than \(\mathcal{T}\).
  Furthermore, the components of \(G-x_0\) are \(G_1-x_0,\ldots, G_\ell-x_0\), so
  \(
  \td(G - x_0)  = \max_{1\le j \le \ell}\td(G_j - x_0).
  \)
  Fix an index \(j \in \{1, \ldots, \ell\}\) such that \(\td(G - x_0) = \td(G_j - x_0)\).
  By the induction hypothesis applied to \(G_j\) and \(x_0\), there is a path \(B_0 x_1 B_1 \cdots x_m B_m\) in \(\mathcal{T}_j\) (and in \(\mathcal{T}\))
  with \(x_0 \in V(B_0)\) and
  \[
    \sum_{i=0}^m \td(B_i - x_i) \ge \td(G _j- x_0) = \td(G-x_0),
  \]
  so the path satisfies the lemma.

  Next, suppose that \(x_0\) is not a cutvertex of \(G\).
  Let \(B_0\) be the unique block of \(G\) containing \(x_0\).
  Let \(\mathcal{T}_1\), \ldots, \(\mathcal{T}_\ell\) be the components of \(\mathcal{T} - B_0\).
  For each \(j\in\{1, \ldots, \ell\}\), let \(y_j\) be the neighbor of \(B_0\) in \(\mathcal{T}\) which
  belongs to \(\mathcal{T}_j\), and let \(G_j\) denote the subgraph of \(G\) obtained as the union of all the blocks of $G$ that are vertices in \(\mathcal{T}_j\). This way, each \(y_j\) is the only common vertex of
  \(B_0\) and \(G_j\), and the block tree of \(G_j\) is either \(\mathcal{T}_j\) or \(\mathcal{T}_j - y_j\)
  depending on whether \(y_j\) has degree at least \(3\) in \(\mathcal{T}\) or not.

  We claim that
  \[
    \td(G-x_0)\le \td(B_0-x_0)+\max_{1 \le j \le \ell} \td(G_j - y_j).
  \]
  We prove this by constructing an elimination forest of \(G-x_0\)
  whose vertex-height is at most the right hand side of the above inequality.
  Take an optimal elimination forest for \(B_0 - x_0\), and for each \(j\in\{1, \ldots, \ell\}\), append all
  trees of an optimal elimination forest for \(G_j - y_j\) right below the vertex \(y_j\) (picturing the root at the top). The vertex-height
  of the resulting forest is at most \(\td(B_0-x_0)+\max_{1 \le j \le \ell} \td(G_j - y_j)\), and it is indeed an elimination forest since the only vertex of \(B_0\) adjacent to vertices of \(G_j - y_j\)
  is \(y_j\).
  Hence, the claimed inequality is satisfied. 

  Fix an index \(j \in \{1, \ldots, \ell\}\) such that \(\td(G-x_0)\le \td(B_0-x_0)+\td(G_j - y_j)\),
  let \(x_1 = y_j\)
  and let the path \(B_1 x_2\cdots x_m B_m\) be the result of applying the induction hypothesis to \(G_j\) and \(x_1\).
  Now, the path \(B_0x_1 \cdots x_m B_m\) satisfies the lemma; indeed, \(x_0 \in V(B_0)\) and
  \[
    \sum_{i=0}^m \td(B_i - x_i) = \td(B_0 - x_0) + \sum_{i=1}^m \td(B_i - x_i) \ge \td(B_0 - x_0) + \td(G_j - y_j) \ge \td(G - x_0). \qedhere
  \]
\end{proof}

In the following lemma, the {\em length} of a path is defined as its number of edges.

\begin{lemma}\label{lem:path-in-graph}
  Let \(G\) be a graph that is either $K_2$ or $2$-connected, and let \(a, b\) be two distinct vertices of \(G\). 
  Then, there exists an \(a\)--\(b\) path in \(G\) of length at least \(\td(G - b)\). 
\end{lemma}
\begin{proof}
  We prove the lemma by induction on the number of vertices in \(G\).
  In the base case \(|V(G)| = 2\), \(G\) consists of a single edge between \(a\) and \(b\),
  \(\td(G - b) = 1\), and \(G\) itself is an \(a\)--\(b\) path of length \(1\), so the lemma holds.
  For the inductive step, $|V(G)| \geq 3$ and thus $G$ is 2-connected. 
  Therefore, \(G - b\) is connected.
  Let the path \(B_0x_1\cdots x_m B_m\) be the result of applying Lemma~\ref{lem:path-in-block-tree} to the graph \(G - b\)
  with \(x_0 = a\). After possibly extending the path, we may assume that \(B_m\) is a leaf in the block tree of \(G - b\)
  (after rooting the block tree at \(B_0\)).
  Since \(G\) is \(2\)-connected, \(x_m\) is not a cutvertex of \(G\), so \(b\) has a neighbor in \(B_m - x_{m}\).
  Let \(x_{m+1}\) be such a neighbor.
  Hence, for each \(i\in \{0, \ldots, m\}\), \(x_i\) and \(x_{i+1}\) are distinct vertices of \(B_i\).
  By the induction hypothesis, for each \(i\in \{0, \ldots, m\}\) we can choose an \(x_i\)--\(x_{i+1}\) path \(P_i\) in \(B_i\) with \(|E(P_i)| \ge \td(B_i - x_i)\).
  By our choice of the path \(B_0x_1\cdots x_m B_m\) and by Lemma~\ref{lem:path-in-block-tree}, 
  \[
    \sum_{i=0}^m |E(P_i)| \ge \sum_{i=0}^m \td(B_i - x_i) \ge \td((G - b) - x_0) = \td(G - \{a, b\}) \ge \td(G - b) - 1.
  \]
  As \(x_0 = a\), the desired \(a\)--\(b\) path can be obtained as the union of the paths \(P_0\), \ldots, \(P_m\) and the edge \(x_{m+1}b\).
\end{proof}

\begin{proof}[Proof of Theorem~\ref{thm:td-vs-circumference}]
  Let \(G\) be a \(2\)-connected graph, and let \(ab\) be any edge of \(G\). 
  Since \(G\) is \(2\)-connected,  \(\td(G) \ge 3\). 
  By Lemma~\ref{lem:path-in-graph}, \(G\) contains an \(a\)--\(b\) path \(P\) of length at least \(\td(G-b) \ge \td(G) - 1\). 
  Note that $P$ does not contain the edge $ab$ since $P$ has length at least $2$.  
  Hence, \(P + ab\) is a cycle of length at least \(\td(G)\) containing $ab$.   
\end{proof}

\section*{Acknowledgments}

This research was started at the Structural Graph Theory workshop in Gułtowy (Poland) in June~2022 organized by Andrzej Grzesik, Marcin Pilipczuk, and Marcin Witkowski. 
We thank the organizers and the other workshop participants for creating a productive working atmosphere. 
We thank Micha\l{} Pilipczuk for pointing out to us the algorithmic application mentioned in the introduction. 
We are grateful to the two anonymous referees for their helpful comments. 

\bibliographystyle{abbrv}
\bibliography{bibliography}

\end{document}